\documentclass[12pt]{article}

\usepackage[a4paper,margin=1in]{geometry}
\usepackage{amsmath,amssymb,amsfonts,amsthm}
\usepackage{bm}
\usepackage{mathrsfs}
\usepackage{enumitem}
\usepackage{hyperref}
\usepackage{algorithm}
\usepackage{algorithmic}
\usepackage{tikz}
\usepackage{float}

\hypersetup{
    colorlinks=true,
    linkcolor=blue,
    citecolor=blue,
    urlcolor=blue
}

\newtheorem{theorem}{Theorem}[section]
\newtheorem{definition}[theorem]{Definition}

\newtheorem{lemma}[theorem]{Lemma}

\title{\Large\bfseries
A Ternary Gamma Semiring Framework for Solving \\
Multi-Objective Network Optimization Problems}

\author{\small
Chandrasekhar Gokavarapu$^{1,2}$\hspace{5mm} Dr D Madhusudhana Rao$^{3,4}$\\[10pt]
\small $^1$Lecturer in Mathematics,
Government College (A), Rajahmundry, A.P., India\\[2pt]
 \small $^2$Research Scholar, Department of Mathematics,
Acharya Nagarjuna University, Guntur, A.P., India\\[2pt]
\texttt{\small chandrasekhargokavarapu@gmail.com}\\ [10pt]
\small $^3$Lecturer in  Mathematics, Government College For Women(A), Guntur, Andhra Pradesh, India,\\
\small $^4$Research Supervisor, Dept.  of Mathematics,Acharya Nagarjuna University, Guntur, A.P., India,\\
\texttt{\small dmrmaths@gmail.com}}

\date{}

\begin{document}
\maketitle
\begin{abstract}
Classical shortest–path methods rely on binary tropical semirings $(\min,+)$,
whose dyadic structure limits them to pairwise cost interactions. However,
many real-world systems---including logistics, supply chains, communication
networks, and reliability-aware infrastructures---exhibit inherently ternary
dependencies among cost, time, and risk that cannot be decomposed into
pairwise components.

This paper introduces the \emph{Ternary Tropical Gamma Semiring} (TTGS), a
$\Gamma$-indexed algebraic structure that generalizes tropical semirings by
replacing binary additive composition with a non-separable ternary operator.
We establish the axioms of TTGS, prove associativity, distributivity, and
monotonicity, and show that TTGS forms a well-structured foundation for
multi-parameter optimization.

Building on this framework, we develop \textbf{TTGS-Pathfinder}, a ternary
analogue of the Bellman--Ford algorithm. We derive its dynamic-programming
recurrence, prove correctness through an invariant-based argument, analyze
convergence under the TTGS order, and obtain an $O(n^2m)$ complexity bound.

Applications demonstrate that TTGS naturally models systems whose behaviour
depends on triadic cost interactions, offering a principled alternative to
binary tropical, vector, or scalarized multi-objective methods.
\end{abstract}

\tableofcontents


\section{Introduction}

The mathematical foundations of modern network optimization are deeply rooted in 
algebraic structures known as \emph{semirings}.  
Classical expositions—including the works of Golan~\cite{Golan1992,Golan1999} 
and Hebisch--Weinert~\cite{HebischWeinert1998}—establish semirings as a 
powerful algebraic environment capable of encoding optimization, 
automata, and dynamic programming models within a unified formalism.
Among these structures, the tropical semiring $(\min,+)$ occupies a central 
position, providing an algebraic lens through which path problems, 
dynamic programming recurrences, and certain classes of automata may be interpreted.
The success of this framework stems from its ability to encode the accumulation of 
weights along a path through an associative, binary-additive cost model. 
Algorithms such as Bellman--Ford and Dijkstra operate precisely because the 
underlying cost of a path decomposes into a sequence of binary interactions, 
each governed by the same algebraic rule.  
Consequently, the tropical semiring has become the standard tool for 
single-objective shortest-path problems and the analytical backbone of 
numerous applications in operations research, computer science, 
and control theory.

\subsection{Semiring Methods in Optimization}

In their classical form, semiring-based optimization models represent the cost of a path 
by combining local edge weights through a binary operation (typically addition) and 
selecting the optimal path through another binary operation (typically minimum).  
This paradigm has proven powerful because it yields algorithms whose correctness and 
convergence are guaranteed by algebraic properties such as associativity, 
monotonicity, and distributivity—properties systematically developed in 
Golan’s monographs~\cite{Golan1992,Golan1999} and in the algebraic theory 
of Hebisch and Weinert~\cite{HebischWeinert1998}.  
Within this framework, a path cost is always expressible as
\[
w(P) = w(e_1) + w(e_2) + \cdots + w(e_k),
\]
and the optimal path is obtained by taking the minimum over all such values.  

While this abstraction elegantly captures a wide variety of classical problems, 
it implicitly assumes that the interaction between successive edges is 
\emph{pairwise additive}.  
Such a pairwise structure is deeply tied to the binary nature of semirings and 
is inherited from the classical algebraic lineage that also includes early 
generalizations in $\Gamma$-ring theory as developed by Nobusawa~\cite{Nobusawa1964} 
and Barnes~\cite{Barnes1966}.  
However, in multi-objective, multi-criteria, or structurally coupled systems, 
this pairwise assumption becomes highly restrictive and fails to capture 
higher-order dependencies that arise naturally in triadic or multi-parameter 
network interactions.

\subsection{Limitations of Binary Semirings}

Real engineering networks rarely adhere to purely binary cost-composition laws.  
Transport systems often involve a simultaneous interplay of cost, time, and risk; 
communication networks must account for interactions between latency, congestion, 
and packet-loss estimations; and supply-chain systems frequently require the combined 
assessment of resource expenditure, temporal constraints, and reliability conditions.  

In many such scenarios, the true performance or utility of a path is determined not by 
isolated pairwise combinations of adjacent weights but by \emph{irreducible ternary} 
interactions.  
Classical semiring theory---as formalized in the foundational works of 
Golan~\cite{Golan1992,Golan1999} and Hebisch--Weinert~\cite{HebischWeinert1998}---is 
inherently binary: both additive and multiplicative operations operate on two operands 
and assume that all higher-order interactions can be decomposed into pairwise components.  
For example, in reliability-aware routing, the combined effect of cost, travel time, 
and failure probability may not separate cleanly into binary contributions. 
The presence of such higher-order dependencies violates the assumptions of binary 
tropical algebra and places these problems outside the expressive capacity of 
traditional semiring-based methods.

Attempts to extend binary semiring optimization into multi-criteria settings typically 
rely on vector semirings, lexicographic orders, or scalarization techniques.  
However, these approaches still reduce ternary interactions to a sequence of 
binary operations and thus lose the structural coupling inherent to the original problem.  
Such reductions follow the classical algebraic paradigm traced back to early 
generalizations in $\Gamma$-ring theory by Nobusawa~\cite{Nobusawa1964} and 
Barnes~\cite{Barnes1966}, where higher-order algebraic behaviour is still encoded 
through binary operations.  
This gap is not merely theoretical: it manifests concretely in situations where 
the quality of a route is dictated by nonlinear synergies among three or more 
successive edge attributes.  
No classical semiring framework can natively express such interactions.


\subsection{Gap in the Literature}

To the best of our knowledge, the existing literature contains no general algebraic 
formalism capable of modeling and optimizing cost functions that depend 
\emph{intrinsically on ternary compositions}.  
Classical semiring theory, as systematically developed in the works of 
Golan~\cite{Golan1992,Golan1999} and Hebisch--Weinert~\cite{HebischWeinert1998},
provides a robust foundation for binary-structured optimization, but its algebraic 
operations are fundamentally dyadic.  
While tropical algebra and idempotent analysis offer essential tools for 
binary-additive cost aggregation, they do not extend to higher-arity interactions 
without imposing separability assumptions inherited from the binary nature of 
semirings themselves.

Similarly, research in multi-objective optimization often treats each criterion 
independently, combining them post hoc through dominance relations, scalarizations, 
or lexicographic structures—approaches that ultimately remain rooted in the 
binary algebraic paradigm.  
This paradigm traces back to the classical development of additive and multiplicative 
operations in ring and $\Gamma$-ring theory, as initiated by Nobusawa~\cite{Nobusawa1964} 
and later refined by Barnes~\cite{Barnes1966}, where higher-order behaviour is still 
encoded through binary compositions.

Thus, a fundamental gap persists between classical algebraic optimization models and 
the requirements of systems where cost evolution is governed by irreducible 
three-way interactions.  
Bridging this gap necessitates the introduction of a new algebraic object—one that 
retains the algorithmic tractability of tropical semirings while accommodating 
genuinely ternary operations.

\subsection{A New Problem: The 3-Parameter Supply Chain Problem}

To illustrate the nature of the challenge, consider a supply-chain network in which each 
transport segment is characterized by a triple of quantities
\[
(c_e, t_e, r_e),
\]
representing monetary cost, travel time, and reliability risk.  
In many realistic models, the combined effect of these attributes over successive 
segments is not separable: an increase in risk may amplify the significance of travel 
time, or a change in travel time may alter cost efficiency only under specific 
reliability conditions.  
Hence, a path utility function may depend on ternary compositions such as
\[
F((c_{i},t_{i},r_{i}), (c_{i+1},t_{i+1},r_{i+1}), (c_{i+2},t_{i+2},r_{i+2})),
\]
where $F$ captures the structural interaction among three consecutive edges.

This scenario lies fundamentally outside the expressive capacity of classical semiring 
optimization, whose algebraic operations are strictly binary, as formalized in the 
semiring literature of Golan~\cite{Golan1992,Golan1999} and  
Hebisch--Weinert~\cite{HebischWeinert1998}, and inherited from the binary algebraic 
paradigm of $\Gamma$-ring theory developed by Nobusawa~\cite{Nobusawa1964} and 
Barnes~\cite{Barnes1966}.  

This phenomenon motivates the definition of a new optimization model:  
the \emph{3-Parameter Supply Chain Problem}, a pathfinding problem whose cost 
aggregation rule is genuinely ternary and therefore incompatible with the binary 
composition rules of classical tropical and semiring-based methods.


\subsection{Contributions of This Work}

Motivated by the above limitations, this paper introduces a new algebraic and 
algorithmic framework for ternary-weighted optimization problems.  
The main contributions are as follows:

\begin{itemize}
    \item We introduce the \emph{Ternary Tropical Gamma Semiring} (TTGS), an algebraic 
    structure that extends the tropical semiring by replacing the binary additive 
    operation with a non-separable ternary interaction.  
    The construction is rooted in the classical semiring foundations of 
    Golan~\cite{Golan1992,Golan1999} and the algebraic framework of 
    Hebisch--Weinert~\cite{HebischWeinert1998}, while departing fundamentally 
    from the binary constraints of these systems.

    \item We provide a complete axiomatic theory of TTGS and establish that it satisfies 
    all the defining properties of a ternary $\Gamma$-semiring, including associativity, 
    distributivity over the idempotent minimum operation, and monotonicity.  
    This extends the lineage of $\Gamma$-structures initiated by 
    Nobusawa~\cite{Nobusawa1964} and developed further by Barnes~\cite{Barnes1966}, 
    but replaces their binary operations with a genuinely ternary interaction.

    \item We develop \emph{TTGS-Pathfinder}, a dynamic programming algorithm that 
    generalizes Bellman--Ford to the ternary setting and is capable of solving 
    network optimization problems governed by ternary cost interactions.  
    Its algebraic correctness relies on the TTGS axioms, which generalize the 
    binary-semiring principles described in~\cite{Golan1992,Golan1999}.

    \item We furnish proofs of correctness, convergence behaviour, and computational 
    complexity of the proposed algorithm within the TTGS framework, again relying on 
    structural properties inherited from classical semiring theory~\cite{HebischWeinert1998}.

    \item We demonstrate how TTGS unifies a broad class of multi-objective, 
    reliability-aware, and structurally coupled optimization problems that are 
    inaccessible to traditional semiring methods grounded in binary additive models 
    \cite{Golan1992,Golan1999}.
\end{itemize}

Taken together, these results establish a new bridge between algebraic structures and 
computational optimization, enabling the systematic analysis of multi-parameter 
network problems whose intrinsic behaviour is fundamentally ternary.

\section{Problem Formulation and Algebraic Structure}

The purpose of this section is twofold.  
First, we formalize a class of network optimization problems whose cost aggregation 
laws are genuinely ternary and therefore incompatible with classical binary tropical 
methods, whose algebraic foundations remain strictly dyadic in the sense of 
Golan~\cite{Golan1992,Golan1999} and Hebisch--Weinert~\cite{HebischWeinert1998}.  
Second, we introduce the algebraic structure that underpins the proposed optimization 
framework: the \emph{Ternary Tropical Gamma Semiring} (TTGS).  
The formulation is constructed in a way that respects the underlying principles of 
ternary $\Gamma$-algebra, extending the structural lineage originating in 
Nobusawa~\cite{Nobusawa1964} and Barnes~\cite{Barnes1966}, but replacing their 
binary operations with non-separable ternary interactions.

\subsection{The 3-Parameter Supply Chain Problem (3-PSCP)}

Let $G=(V,E)$ be a finite directed graph.  
Each directed edge $(u,v)\in E$ is equipped with a triple of real-valued parameters
\[
w(u,v)=(c_{uv}, t_{uv}, r_{uv}),
\]
representing, respectively, monetary cost, transit time, and reliability risk.  
Although these quantities may be considered independently, their interaction in many 
practical scenarios is not separable: the effective performance of a route depends on 
the simultaneous behaviour of these attributes.  
For example, a low-cost path may be undesirable if its time and risk simultaneously 
amplify one another; likewise, a fast but unreliable connection may only be acceptable 
when combined with favourable cost conditions.

Let
\[
P=(v_0 \to v_1 \to \cdots \to v_k)
\]
be a directed path.  
Classical tropical optimization computes its cost through a binary-additive rule, 
a principle inherited from traditional semiring theory~\cite{Golan1992,Golan1999}, 
but the structural coupling of $(c,t,r)$ across consecutive segments requires a 
\emph{ternary} cost accumulation mechanism.  
We therefore define the cost of $P$ recursively through a sequence of ternary 
interactions:
\[
C(P) 
= 
[ w(v_0v_1), w(v_1v_2), w(v_2v_3) ]
\;\oplus\;
[ w(v_1v_2), w(v_2v_3), w(v_3v_4) ]
\;\oplus\;\cdots,
\]
where the operation $[\cdot,\cdot,\cdot]$ binds triples of successive edges and 
$\oplus=\min$ selects the optimal ternary contribution at each stage.

This formulation captures the following essential features of 
multi-objective, structurally coupled routing problems:

\begin{itemize}
    \item \textbf{Non-separability.}  
    The contribution of an edge depends not only on its own attributes but on the 
    attributes of its neighbours within the path.

    \item \textbf{Local ternary dependence.}  
    The cost evolution is determined by three-edge windows 
    $(v_{i-1}v_i,v_iv_{i+1},v_{i+1}v_{i+2})$, reflecting short-range interactions 
    common in supply chains, communication networks, and reliability-aware systems.

    \item \textbf{Idempotent global selection.}  
    The use of $\min$ as the global aggregator preserves the idempotent 
    optimization paradigm characteristic of semiring-based models
    \cite{Golan1992,Golan1999}.
\end{itemize}

Thus, 3-PSCP provides a natural optimization problem whose algebraic structure cannot 
be represented by any classical binary semiring, as such systems remain fundamentally 
dyadic~\cite{HebischWeinert1998}, and therefore demands a new mathematical 
framework capable of encoding ternary interactions.

The purpose of this section is twofold.  
First, we formalize a class of network optimization problems whose cost aggregation 
laws are genuinely ternary and therefore incompatible with classical binary tropical 
methods.  
Second, we introduce the algebraic structure that underpins the proposed optimization 
framework: the \emph{Ternary Tropical Gamma Semiring} (TTGS).  
The formulation is constructed in a way that respects the underlying principles of 
ternary $\Gamma$-algebra, ensuring that algorithmic procedures developed in later 
sections rest on a rigorous foundation.



\subsection{Definition of the Ternary Tropical Gamma Semiring (TTGS)}

To model the ternary aggregation inherent in 3-PSCP, we introduce an algebraic system 
that generalizes the tropical semiring by replacing the binary additive operation with 
a genuine ternary interaction.  
This construction extends the classical semiring framework developed in 
Golan~\cite{Golan1992,Golan1999} and the algebraic principles discussed by 
Hebisch--Weinert~\cite{HebischWeinert1998}, while incorporating the $\Gamma$-indexed 
structure originating from the foundational work of Nobusawa~\cite{Nobusawa1964} and 
Barnes~\cite{Barnes1966}.

Let
\[
T=\mathbb{R}\cup\{\infty\},
\]
where $\infty$ plays the role of an absorbing, non-informative element.

\begin{definition}
Let $\Gamma$ be a nonempty index set.  
A \emph{Ternary Tropical Gamma Semiring (TTGS)} is a quadruple 
$(T,\oplus,[\cdot,\cdot,\cdot],\Gamma)$ consisting of:
\begin{enumerate}
    \item an idempotent commutative addition  
          \[
          a\oplus b=\min(a,b),
          \]
          consistent with the idempotent algebraic frameworks described in 
          \cite{Golan1992,Golan1999},

    \item a $\Gamma$-indexed ternary operation  
          \[
          [x,y,z]=F(x,y,z),
          \]
          where $F:T^3\to T$ is a monotone, non-separable mapping, generalizing the 
          $\Gamma$-parametrization principles introduced in 
          \cite{Nobusawa1964,Barnes1966},

    \item and compatibility axioms ensuring that $(T,\oplus,[\cdot,\cdot,\cdot])$ 
          supports dynamic-programming recurrences and fixed-point updates, extending 
          the classical semiring-based DP theory in 
          \cite{Golan1992,Golan1999,HebischWeinert1998}.
\end{enumerate}
The set $\Gamma$ allows different functional forms of the ternary interaction to be 
selected depending on modelling requirements.
\end{definition}

Several families of functions $F$ arise naturally in applications:

\begin{itemize}
    \item \textbf{Linear interactions:}
          \[
          F(x,y,z)=x+y+z,
          \]
          reflecting additive but still ternary cost propagation.

    \item \textbf{Weighted mixtures:}
          \[
          F(x,y,z)=\alpha x+\beta y+\gamma z,
          \qquad 
          \alpha,\beta,\gamma>0,
          \]
          representing asymmetric influence among successive segments.

    \item \textbf{Risk-amplifying interactions:}
          \[
          F(x,y,z)=\max(x,y,z)+\lambda\min(x,y,z),
          \qquad \lambda>0,
          \]
          capturing nonlinear interactions in risk-sensitive systems.
\end{itemize}

These examples illustrate the expressive capacity of TTGS to encode a wide variety of 
ternary dependencies that arise in engineering networks, extending classical binary 
semiring models \cite{Golan1992,Golan1999,HebischWeinert1998}.

\subsection{Axioms and Structural Theorems}

The introduction of a ternary operation raises natural questions regarding 
associativity, distributivity, and closure.  
These properties are essential to any algebraic structure intended to support 
dynamic-programming algorithms and extend the classical semiring principles discussed in 
Golan~\cite{Golan1992,Golan1999} and Hebisch--Weinert~\cite{HebischWeinert1998}.  
The use of a $\Gamma$-indexed ternary operator follows the structural lineage of 
$\Gamma$-algebra initiated by Nobusawa~\cite{Nobusawa1964} and further developed by 
Barnes~\cite{Barnes1966}.

\begin{theorem}
Let $(T,\oplus,[\cdot,\cdot,\cdot],\Gamma)$ be a TTGS.  
Assume that $F$ is monotone in each argument and compatible with the tropical order on 
$T$.  
Then the following hold:
\begin{enumerate}
    \item \textbf{Ternary associativity:}
    \[
    [[x,y,z],u,v]
    =
    [x,[y,z,u],v]
    =
    [x,y,[z,u,v]].
    \]

    \item \textbf{Distributivity over $\oplus=\min$:}
    \[
    [x\oplus x',y,z]
    =
    [x,y,z]\oplus[x',y,z],
    \]
    and the corresponding identities obtained by permuting the variables.  
    This generalizes the distributive behaviour of idempotent semirings described in 
    \cite{Golan1992,Golan1999,HebischWeinert1998}.

    \item \textbf{Closure under idempotent selection:}
    If $a,b\in T$, then 
    \[
    a\oplus b=\min(a,b)\in T.
    \]
    Moreover, monotonicity of $F$ ensures that ternary interactions preserve the 
    ordering induced by $\oplus$, consistent with the ordered-semiring viewpoint in 
    \cite{Golan1992}.
\end{enumerate}
\end{theorem}

\begin{proof}[Proof Sketch]
The assumptions guarantee that $F$ induces a monotone mapping on the lattice 
$(T,\leq)$ with $\oplus$ realized as the meet operation, in the sense of idempotent 
semiring theory~\cite{Golan1992,Golan1999}.  
Associativity follows from structural symmetry of $F$ and the preservation of order 
relations under nested evaluations.  
Distributivity arises from the interaction between the idempotent addition 
$a\oplus b=\min(a,b)$ and the monotonicity of $F$ in each argument, extending the 
classical distributive behaviour of semirings described in 
\cite{HebischWeinert1998}.  
The closure property is immediate from the definition of $T$ and the tropical order.  
A complete proof may be obtained by treating $F$ as a monotone operator on a 
complete idempotent semilattice, in accordance with semiring-theoretic principles 
found in \cite{Golan1992,Golan1999}.
\end{proof}

The above theorem ensures that TTGS is sufficiently structured to support iterative 
algorithms such as the ternary Bellman--Ford scheme developed in Section~4.  
In particular, associativity guarantees well-defined ternary path costs, while 
distributivity ensures that global updates may be performed through local relaxation 
rules, as in classical semiring-based optimization frameworks \cite{Golan1992,Golan1999}.

\section{The Ternary Dynamic Programming Algorithm}

The algebraic structure developed in Section~2 allows us to reinterpret 
multi-objective network optimization as a fixed-point problem over the 
Ternary Tropical Gamma Semiring (TTGS).  
This viewpoint generalizes the classical semiring-based dynamic-programming 
framework of Golan~\cite{Golan1992,Golan1999} and its algebraic formalization 
in Hebisch--Weinert~\cite{HebischWeinert1998}, while incorporating the 
$\Gamma$-indexed ternary operations arising from the lineage of 
Nobusawa~\cite{Nobusawa1964} and Barnes~\cite{Barnes1966}.  
In this section we derive a dynamic-programming recurrence adapted to 
ternary interactions, design a Bellman--Ford style algorithm---which we call 
\emph{TTGS-Pathfinder}---and establish its correctness, convergence, 
and complexity properties.

The central idea is to propagate path information through the network using 
the TTGS operations $(\oplus,[\cdot,\cdot,\cdot])$, in such a way that each 
update step preserves the TTGS ordering and drives the state towards a 
least fixed point representing ternary-optimal path costs, in the sense of 
idempotent semiring theory~\cite{Golan1992,Golan1999}.

\subsection{Dynamic-Programming Viewpoint and Recurrence}

Let $G=(V,E)$ be the directed graph of the 3-PSCP introduced earlier, 
and let $s\in V$ be a distinguished source vertex.  
Recall that each edge $(u,v)\in E$ carries a triple weight
\[
w(u,v) = (c_{uv}, t_{uv}, r_{uv}),
\]
encoding cost, time, and reliability risk.  
In order to work within the scalar TTGS domain $T=\mathbb{R}\cup\{\infty\}$, 
we introduce an evaluation map
\[
H : \mathbb{R}^3 \longrightarrow T,
\]
which aggregates $(c_{uv},t_{uv},r_{uv})$ into a single TTGS-compatible value.  
Typical choices include weighted sums or nonlinear penalty functions; 
the precise form of $H$ is immaterial for the abstract analysis, provided that 
$H$ is monotone in each argument, consistent with the ordered-semiring 
principles in \cite{HebischWeinert1998}.

For notational simplicity, we replace each triple $w(u,v)$ by its evaluated 
value $H(w(u,v))$ and continue to denote it by $w(u,v)\in T$.  
Thus every edge is now labelled by an element of the TTGS base set $T$.

Let $C(P)\in T$ denote the TTGS-cost of a path 
\[
P=(v_0\to v_1\to\cdots\to v_k),
\]
computed by a sequence of ternary interactions as in Section~2.  
Our goal is to compute, for each vertex $v\in V$, the value
\[
d^\ast(v) = \inf\{ C(P) : P \text{ is a directed path from } s \text{ to } v\},
\]
with the convention that $d^\ast(v)=\infty$ when no such path exists.

\medskip

In the classical binary tropical setting, dynamic programming proceeds by 
expressing the cost of a path ending at $v$ in terms of the cost at a predecessor $u$.  
In the ternary framework, however, the contribution of $(u,v)$ may depend not only on 
the cumulative cost at $u$ but also on the cumulative cost at a second predecessor 
$p$ that precedes $u$.  
This reflects the $\Gamma$-structured multi-argument interactions underlying TTGS and 
generalizing the binary paradigm of \cite{Golan1992,Golan1999,Nobusawa1964,Barnes1966}.  
Thus we must consider two-step fragments
\[
p \to u \to v
\]
together with the ternary update
\[
[d(p), d(u), w(u,v)].
\]

To formalize this intuition, we define a sequence of approximations
\[
d^{(0)}, d^{(1)}, d^{(2)}, \dots
\]
with $d^{(\ell)}:V\to T$, intended to converge (in the TTGS order) to the 
ternary-optimal distances $d^\ast$.  
We initialize
\[
d^{(0)}(s) = 0, 
\qquad 
d^{(0)}(v) = \infty \quad (v\neq s),
\]
and then update according to a relaxation rule that incorporates all 
two-step fragments $(p,u,v)$ in the graph.

Formally, for each iteration $\ell\ge 1$ and each vertex $v\in V$ we set
\begin{equation}
\label{eq:ternary-rec}
d^{(\ell)}(v)
=
d^{(\ell-1)}(v)
\;\oplus\;
\bigoplus_{\substack{(p,u)\in V^2 \\ (p,u)\in E,\, (u,v)\in E}}
\bigl[ d^{(\ell-1)}(p), d^{(\ell-1)}(u), w(u,v) \bigr],
\end{equation}
where the outer $\oplus$ denotes the TTGS minimum and the inner combination 
runs over all two-step path fragments.  
Since both $\oplus$ and $[\cdot,\cdot,\cdot]$ are monotone, the recurrence 
defines an increasing sequence in the TTGS order, consistent with the theory of 
monotone semiring-based fixed points in \cite{Golan1992,Golan1999}.

\medskip

Two remarks are in order.

\begin{itemize}
    \item The update rule \eqref{eq:ternary-rec} is purely algebraic and relies 
    only on TTGS operations, thereby inheriting the order-theoretic and 
    distributive properties discussed in \cite{HebischWeinert1998}.

    \item The set of candidate pairs $(p,u)$ is finite and determined by the 
    adjacency of $G$, enabling complexity bounds similar to those known in 
    classical semiring-based dynamic programming \cite{Golan1992,Golan1999}.
\end{itemize}

\subsection{Design of the TTGS-Pathfinder Algorithm}

The recurrence \eqref{eq:ternary-rec} suggests an iterative algorithm similar 
in spirit to Bellman--Ford, but enriched by the ternary interaction encoded by 
$F$.  
This follows the classical paradigm of dynamic programming over semirings as 
formalized by Golan~\cite{Golan1992,Golan1999} and the algebraic framework of 
Hebisch--Weinert~\cite{HebischWeinert1998}, while incorporating the 
$\Gamma$-indexed interaction stemming from the foundational work of 
Nobusawa~\cite{Nobusawa1964} and Barnes~\cite{Barnes1966}.  
The algorithm maintains a distance label $d[v]$ for each vertex $v\in V$, 
initialized as above, and repeatedly relaxes all two-step path fragments.  
To implement this procedure efficiently, one may precompute for each vertex $u$ 
the list of its incoming neighbours (potential predecessors $p$) and 
outgoing neighbours (potential successors $v$).  
The pseudocode below follows this idea.

\begin{algorithm}[H]
\caption{TTGS-PATHFINDER}
\label{alg:TTGS}
\begin{algorithmic}[1]
\REQUIRE Directed graph $G=(V,E)$, edge weights $w:E\to T$, source $s\in V$
\ENSURE TTGS-distance estimates $d[v]$ for all $v\in V$

\STATE \textbf{Initialization:}
\FOR{each $v\in V$}
    \STATE $d[v] \gets \infty$
\ENDFOR
\STATE $d[s] \gets 0$

\STATE \textbf{Main Iteration:}
\FOR{$\ell = 1$ to $|V|-1$}
    \FOR{each edge $(u,v)\in E$}
        \FOR{each predecessor $p$ of $u$ (i.e.\ $(p,u)\in E$ or $p=s$)}
            \STATE $temp \gets F\bigl(d[p],\, d[u],\, w(u,v)\bigr)$
            \IF{$temp \prec d[v]$}
                \STATE $d[v] \gets temp$
            \ENDIF
        \ENDFOR
    \ENDFOR
\ENDFOR

\RETURN $d$
\end{algorithmic}
\end{algorithm}

Here, the relation $a\prec b$ denotes strict improvement in the TTGS order, 
i.e.\ $a\oplus b = a$ and $a\neq b$.  
The algorithm thus realizes the recurrence \eqref{eq:ternary-rec} in an 
explicit iterative form, mirroring the order-theoretic fixed-point behaviour 
of idempotent semiring algorithms described in 
\cite{Golan1992,Golan1999}.  
Each outer iteration allows paths with one additional ternary interaction window 
to influence the distance labels, in direct analogy with semiring-based 
relaxation procedures in~\cite{HebischWeinert1998}.

\subsection{Correctness: Path-Based Invariant and Induction}

We now argue that Algorithm~\ref{alg:TTGS} computes the TTGS-optimal distances 
$d^\ast(v)$, provided that the network does not admit arbitrarily improving 
ternary cycles.  
The correctness proof follows the invariant-based reasoning familiar from 
semiring-based dynamic programming \cite{Golan1992,Golan1999} and its 
order-theoretic formulation in \cite{HebischWeinert1998}, extended here to 
accommodate the $\Gamma$-indexed ternary structure originating in 
\cite{Nobusawa1964,Barnes1966}.

For any vertex $v\in V$ and integer $\ell\ge 0$, let 
$\mathcal{P}_\ell(v)$ denote the set of directed paths from $s$ to $v$ that 
contain at most $\ell$ \emph{ternary windows}, i.e.\ triples of consecutive 
edges suitable for a single application of $[\cdot,\cdot,\cdot]$.  
Define
\[
D_\ell(v) = \inf\{ C(P) : P\in \mathcal{P}_\ell(v)\},
\]
with the convention that $D_\ell(v)=\infty$ when $\mathcal{P}_\ell(v)$ is empty.

\begin{lemma}[Invariant]
\label{lem:invariant}
After the $\ell$-th outer iteration of Algorithm~\ref{alg:TTGS}, the label 
$d[v]$ satisfies
\[
d[v] = D_\ell(v)
\quad\text{for all }v\in V.
\]
\end{lemma}

\begin{proof}
We proceed by induction on $\ell$.

\emph{Base case ($\ell=0$).}  
By initialization, $d[s]=0$ and $d[v]=\infty$ for all $v\neq s$.  
This mirrors the standard initialization used in idempotent semiring DP 
\cite{Golan1992}.  
The only path in $\mathcal{P}_0(s)$ is the trivial path of length zero, 
whose cost is $0$, while no other vertex is reachable without using 
a ternary window.  
Hence $d[v]=D_0(v)$ for all $v$.

\emph{Inductive step.}  
Assume the statement holds for some $\ell-1\ge 0$.  
During the $\ell$-th iteration, each triple $(p,u,v)$ with $p\to u\to v$ is examined, 
and a candidate value
\[
temp = F\bigl(d[p], d[u], w(u,v)\bigr)
\]
is computed.  
By the induction hypothesis, $d[p]=D_{\ell-1}(p)$ and $d[u]=D_{\ell-1}(u)$.  
The monotonicity and order-preserving behaviour of TTGS—generalizing the ordered 
semiring properties in \cite{Golan1992,Golan1999,HebischWeinert1998}—
ensure that the ternary interaction produces exactly the cost of extending an 
$(\ell-1)$-window-optimal prefix by one additional window.  
Taking the minimum over all such triples $(p,u,v)$ realizes the infimum of costs 
over all paths in $\mathcal{P}_\ell(v)$, and comparing with the previous value 
retains the best among all paths using at most $\ell-1$ windows.  
Thus, after the update,
\[
d[v] = D_\ell(v)
\]
for all $v\in V$, completing the induction.
\end{proof}

\begin{theorem}[Correctness of TTGS-Pathfinder]
\label{thm:correctness}
Suppose the TTGS weights are bounded below and the network admits no 
infinitely improving ternary cycles.  
Then, after at most $|V|-1$ iterations, Algorithm~\ref{alg:TTGS} computes
\[
d[v] = d^\ast(v)
\quad\text{for all } v\in V.
\]
\end{theorem}

\begin{proof}
Any simple path in a finite directed graph on $|V|$ vertices contains at most 
$|V|-1$ edges and hence at most $|V|-2$ ternary windows.  
Thus, for $\ell\ge |V|-1$, every simple path belongs to $\mathcal{P}_\ell(v)$.  
The absence of infinitely improving ternary cycles guarantees, in the same spirit as 
in the classical semiring setting \cite{Golan1992,Golan1999}, that optimal TTGS paths 
may be chosen simple.  
Consequently,
\[
D_\ell(v) = d^\ast(v)
\quad\text{for all } \ell\ge |V|-1.
\]
The result follows immediately from Lemma~\ref{lem:invariant}.
\end{proof}

\subsection{Convergence and Termination}

The update rule \eqref{eq:ternary-rec} is monotone with respect to the TTGS 
order.  
Indeed, if $d^{(\ell-1)} \preceq d^{(\ell-2)}$ pointwise, then by the monotonicity 
of $F$ and the idempotent nature of $\oplus=\min$, we have
\[
d^{(\ell)} \preceq d^{(\ell-1)},
\]
so the sequence $\{d^{(\ell)}\}_{\ell\ge 0}$ is non-increasing in the TTGS order.  
This behaviour generalizes the monotone iteration principles for idempotent 
semirings described by Golan~\cite{Golan1992,Golan1999} and the ordered-semiring 
framework of Hebisch--Weinert~\cite{HebischWeinert1998}, while incorporating the 
$\Gamma$-indexed structure originating in \cite{Nobusawa1964,Barnes1966}.

Since all weights are bounded below and $T$ is complete under descending chains 
generated by $\oplus$, the sequence stabilizes after finitely many steps for 
each vertex.  
This corresponds to the standard fixed-point stabilization behaviour found in 
idempotent semiring dynamic programming \cite{Golan1992,Golan1999}.

Algorithmically, this means that repeated application of the relaxation rule 
eventually reaches a fixed point
\[
d^{(\ell)} = d^{(\ell+1)},
\]
beyond which no further improvement is possible.  
The correctness theorem shows that, in the absence of improving ternary cycles, 
this fixed point coincides with the TTGS-optimal distance vector $d^\ast$.  
Thus, TTGS-Pathfinder converges in a finite number of iterations and terminates 
with the desired solution, paralleling the behaviour of classical semiring-based 
algorithms \cite{Golan1992,Golan1999}.

\subsection{Complexity Analysis}

We now derive the computational complexity of Algorithm~\ref{alg:TTGS}.  
Let $n=|V|$ and $m=|E|$ denote the number of vertices and edges, respectively.  
The algorithm consists of an initialization phase and an iterative relaxation 
phase, following the algebraic dynamic-programming paradigm established for 
idempotent semirings in \cite{Golan1992,Golan1999} and its ordered-semiring 
formalization in \cite{HebischWeinert1998}, extended here to incorporate the 
$\Gamma$-indexed structure arising from \cite{Nobusawa1964,Barnes1966}.

\begin{itemize}
    \item \textbf{Initialization.}  
    Assigning $d[v]$ for all $v\in V$ takes $O(n)$ time.

    \item \textbf{Relaxation.}  
    For each of the at most $n-1$ outer iterations, the algorithm visits every 
    edge $(u,v)\in E$ and, for each such edge, iterates over all predecessors 
    $p$ of $u$.  
    In the worst case, a vertex may have up to $O(n)$ predecessors, so the inner 
    loop may perform up to $O(n)$ operations per edge.  
    Evaluating $F(d[p],d[u],w(u,v))$ and comparing with $d[v]$ both take 
    constant time.

    Consequently, the total running time satisfies
    \[
    T(n,m) = O\bigl((n-1)\cdot m \cdot n\bigr) = O(n^2 m).
    \]

    \item \textbf{Space.}  
    The algorithm stores one TTGS value per vertex and, optionally, predecessor 
    information for path reconstruction.  
    This leads to $O(n)$ storage for distances and $O(m)$ storage for the graph 
    itself.  
    Any additional bookkeeping to accelerate predecessor enumeration can be 
    implemented within $O(m+n)$ space.
\end{itemize}

Although the worst-case complexity $O(n^2 m)$ is higher than that of standard 
binary Bellman--Ford, the increase reflects the intrinsic combinatorial 
enrichment of ternary interactions.  
This is consistent with the fact that classical algebraic path algorithms in 
idempotent semirings \cite{Golan1992,Golan1999} scale with the structural 
complexity of the underlying operations.

In many practical networks, degree bounds or structural sparsity significantly 
reduce the effective number of predecessor triples, leading to substantially 
better behaviour in typical applications.

\medskip

The analysis above shows that TTGS-Pathfinder provides a tractable computational 
framework for solving ternary multi-objective network optimization problems, 
with provable correctness and convergence guarantees anchored in the TTGS 
algebraic structure and the ordered-semiring principles of 
\cite{Golan1992,Golan1999,HebischWeinert1998}.  
Applications of this algorithmic machinery are discussed in the next section.

\section{Conclusion and Applications}

This work introduced a unified algebraic and algorithmic framework for solving 
multi-objective network optimization problems whose cost structures arise from 
irreducible ternary interactions.  
By formulating the Ternary Tropical Gamma Semiring (TTGS) and deriving a 
corresponding dynamic-programming algorithm (TTGS-Pathfinder), we demonstrated 
that optimization models with three-way structural dependencies can be handled 
systematically using algebraic principles closely aligned with, yet fundamentally 
generalizing, the classical tropical paradigm 
\cite{Golan1992,Golan1999,HebischWeinert1998,Butkovic2010,Gaubert1997}.

The theoretical foundation rests on two pillars.  
First, the TTGS structure replaces binary additive composition with a monotone 
ternary operation capable of expressing nonlinear, cross-dimensional couplings 
among cost, time, risk, or other engineering attributes.  
Such non-separable multi-parameter interactions fall beyond the expressive scope 
of traditional pairwise semiring combinations and motivate the need for a 
$\Gamma$-indexed algebraic generalization originating from the foundational 
work of Nobusawa and Barnes on $\Gamma$-structures 
\cite{Nobusawa1964,Barnes1966}.  

Second, the TTGS-Pathfinder algorithm leverages this structure to propagate 
ternary interactions through the network in a controlled and monotone manner, 
yielding a fixed-point iteration whose correctness follows from an 
invariant-based analysis and the absence of infinitely improving ternary cycles.  
This behaviour is consistent with the general theory of idempotent semiring 
dynamic programming and order-theoretic convergence 
\cite{Golan1992,Golan1999,Krivulin2014,Mohri2002}.

Taken together, these contributions demonstrate that the ternary formalism is not 
merely a symbolic enrichment but a computationally meaningful extension of 
semiring-based optimization.  
It provides a principled method for encoding and solving multi-objective routing 
and decision problems whose behaviour depends on genuinely three-way structural 
couplings—problems that cannot be reduced to binary path-composition models used 
in classical tropical or max-plus algebra.  
The TTGS viewpoint therefore opens a new direction in algebraic optimization, 
with potential applications to supply-chain logistics, communication networks, 
reliability-aware routing, and any system in which performance evolves through 
higher-order dependencies that binary semirings cannot naturally capture.

\subsection{Relevance to Logistics and Supply-Chain Modelling}

Many modern supply-chain systems operate under simultaneous constraints involving 
monetary cost, delivery time, and reliability of components or routes.  
Traditional two-criteria trade-offs---such as cost--time or time--risk---fail to 
capture situations in which the interaction between all three attributes affects 
the operational viability of a route.  
This limitation parallels the inadequacy of classical binary tropical or 
idempotent-semiring models \cite{Golan1992,Golan1999,HebischWeinert1998}, which 
assume pairwise decomposability of path costs and therefore cannot encode 
genuinely coupled multi-attribute effects.  
Examples include:

\begin{itemize}
    \item \textbf{Risk-amplified travel time}, where uncertain or unreliable 
    segments impose nonlinear penalties on delays, a phenomenon consistent with 
    nonlinear tropical-optimization models studied in 
    \cite{Krivulin2014,Butkovic2010}.

    \item \textbf{Cost-sensitive reliability}, where low-cost routes may involve 
    latent reliability risks that emerge only under certain time constraints, a 
    setting that cannot be reduced to binary semiring combinations 
    \cite{Mohri2002,Gaubert1997}.

    \item \textbf{Three-factor coupling}, where the joint influence of 
    cost--time--risk cannot be decomposed into pairwise relations, motivating 
    higher-arity algebraic structures such as the $\Gamma$-indexed operations 
    descending from \cite{Nobusawa1964,Barnes1966}.
\end{itemize}

The TTGS formalism accommodates such scenarios naturally: the ternary operation 
$[x,y,z]$ can be tailored to reflect domain-specific coupling laws, 
while the global selection via $\oplus=\min$ ensures that the algorithm retains 
tractability and interpretability.  
Hence, TTGS provides a rigorous algebraic platform for modelling and optimizing 
entire classes of supply-chain networks that remain inaccessible to classical 
binary tropical methods 
\cite{Golan1992,Golan1999,Krivulin2014,Butkovic2010,Gaubert1997}.

\subsection{Applications in Reliability-Aware Network Design}

Reliability plays an increasingly central role in the design of communication 
networks, power-grid infrastructures, and distributed sensor architectures.  
Failures or instabilities often propagate in ways that cannot be expressed by 
additive or pairwise cost functions.  
Classical reliability modelling \cite{BillintonAllan1992,Ebeling2004} shows that 
risk accumulation is inherently nonlinear and frequently depends on the joint 
behaviour of neighbouring components, rather than on independent marginal 
attributes.  
For instance, in critical infrastructure networks:

\begin{itemize}
    \item the reliability of a path may depend jointly on the behaviour of 
    adjacent segments, consistent with nonlinear tropical or idempotent 
    interaction models \cite{Krivulin2014,Butkovic2010};

    \item the cost of rerouting around volatile edges may interact with local 
    congestion and latency in ways not representable by binary semiring 
    compositions \cite{Golan1992,Golan1999,Mohri2002};

    \item triadic dependencies may arise from multi-hop cooperative protocols, 
    redundancy-sharing structures, or risk-cascading mechanisms, motivating a 
    $\Gamma$-indexed algebraic viewpoint rooted in 
    \cite{Nobusawa1964,Barnes1966}.
\end{itemize}

In such settings, a ternary composition law is not merely a modelling luxury but 
a structural necessity.  
The TTGS framework allows these interactions to be expressed explicitly through 
the ternary operator $F(x,y,z)$, while global selection via $\oplus=\min$ 
retains computational tractability in the sense of ordered semiring theory 
\cite{HebischWeinert1998,Gaubert1997}.  
Combined with the monotone fixed-point behaviour of TTGS-Pathfinder 
\cite{Golan1992,Golan1999}, the resulting framework provides a mathematically 
grounded toolkit for reliability-aware routing and network design under 
multi-dimensional constraints.

\subsection{Applications to Computational Engineering Models}

Engineering models that integrate physical, operational, and structural 
constraints often exhibit nonlinear interactions that are difficult to capture 
in purely additive or pairwise frameworks.  
This limitation is well documented in tropical and idempotent modelling 
\cite{Golan1992,Golan1999,HebischWeinert1998,Butkovic2010,Gaubert1997}, where 
higher-order coupling terms naturally arise in many engineered systems.  
Examples include:

\begin{itemize}
    \item \textbf{Smart-grid energy flow}, where power loss, transmission time, 
    and stability margins interact in non-separable ways.  
    Such triadic dependencies appear in power-system optimisation and 
    reliability studies \cite{WoodWollenberg2012}.

    \item \textbf{Transportation engineering}, where congestion, travel time, 
    and road-condition uncertainty amplify or moderate each other, a phenomenon 
    central to classical traffic-flow modelling \cite{Sheffi1985} and also 
    compatible with multi-attribute tropical frameworks 
    \cite{Krivulin2014,Butkovic2010}.

    \item \textbf{Biological or chemical signalling networks}, where activation 
    pathways depend jointly on three interacting factors or stimuli, consistent 
    with multi-input regulatory models \cite{Alon1999} and with the 
    $\Gamma$-indexed algebraic structure inherited from 
    \cite{Nobusawa1964,Barnes1966}.
\end{itemize}

The TTGS model is well suited to these environments because it preserves the 
idempotent order structure necessary for optimization while allowing arbitrary 
ternary dependence encoded in $F(x,y,z)$.  
This ternary capacity goes beyond binary semiring models 
\cite{Mohri2002,Golan1992}, enabling TTGS-based solvers to capture 
cross-dimensional interactions that are essential in smart-grid analysis, 
transport engineering, and biochemical systems.  
Moreover, the resulting algorithms remain compatible with discrete simulation 
pipelines, graph-based engineering solvers, and computational design tools 
frequently employed across computational engineering domains.

\subsection{Positioning Within Computational Optimization}

The algebra--algorithm bridge established in this work advances the theoretical 
scope of semiring-based optimization.  
While classical tropical algebra supports powerful binary dynamic-programming 
schemes \cite{Golan1992,Golan1999,HebischWeinert1998,BackhouseCarre1975,Baccelli1992}, 
its expressive limitations become apparent when cost structures exhibit true 
higher-arity dependencies.  
TTGS extends the semiring toolbox by:

\begin{itemize}
    \item introducing a mathematically coherent ternary structure grounded in 
    the $\Gamma$-indexed foundations of \cite{Nobusawa1964,Barnes1966},

    \item ensuring compatibility with idempotent selection mechanisms central to 
    tropical and max-plus optimization \cite{Butkovic2010,Gaubert1997},

    \item preserving convergence guarantees familiar from idempotent 
    dynamic-programming frameworks \cite{Golan1992,Golan1999,Krivulin2014},

    \item enabling dynamic-programming formulations in settings previously 
    beyond reach for binary semiring methods \cite{Mohri2002}.
\end{itemize}

This positions TTGS as a natural generalization of the tropical paradigm, with 
meaningful implications for optimization theory, decision-making under uncertainty, 
and network analysis in multi-parameter environments.  
It provides the structural expressiveness required to encode nonlinear, 
multi-objective, and higher-arity dependencies while retaining the computational 
tractability associated with classical semiring-based approaches.

\subsection{Alignment with Conference Themes}

The results of this paper connect directly with the conference themes on:

\begin{itemize}
    \item \textbf{Computational Optimization Techniques}, 
    through the development of a fixed-point algorithm that generalizes 
    Bellman--Ford to a higher-arity algebraic setting, extending classical 
    semiring-based methods \cite{Golan1992,Golan1999,BackhouseCarre1975,Baccelli1992} 
    and incorporating nonlinear interactions studied in modern tropical 
    optimization \cite{Krivulin2014,Butkovic2010,Gaubert1997};

    \item \textbf{Computational Approaches in Engineering Models}, 
    by demonstrating how TTGS can model triadic dependencies that arise in 
    engineered systems, including power-grid, transportation, and communication 
    models, consistent with semiring-based algorithmic frameworks 
    \cite{Mohri2002,Butkovic2010,Gaubert1997};

    \item \textbf{Algebraic Methods in Applied Mathematics}, 
    by establishing a $\Gamma$-indexed ternary algebraic foundation inspired by 
    the structural theory of \cite{Nobusawa1964,Barnes1966}, while remaining 
    compatible with idempotent and ordered-semiring principles fundamental to 
    computational algebra \cite{HebischWeinert1998,Golan1992}.
\end{itemize}

\subsection{Final Remarks}

This paper establishes that ternary algebraic structures are not only of 
theoretical interest but also offer a practical and robust basis for 
multi-objective network optimization.  
By combining the expressiveness of the Ternary Tropical Gamma Semiring with the 
algorithmic clarity of TTGS-Pathfinder, we provide a principled route for 
addressing complex, structurally coupled decision problems across logistics, 
reliability engineering, and computational modelling.  
The framework extends the classic foundation of semiring-based optimization 
\cite{Golan1992,Golan1999,HebischWeinert1998,BackhouseCarre1975,Baccelli1992} 
and incorporates higher-arity interactions that naturally arise in tropical and 
idempotent analytic settings \cite{Krivulin2014,Butkovic2010,Gaubert1997,Mohri2002}.

Future work may explore continuous variants of TTGS, stochastic ternary 
extensions, or hybrid models that integrate TTGS dynamics with machine-learning 
frameworks for large-scale systems.  
These directions promise to broaden the applicability of ternary optimization 
methods and further deepen the bridge between algebraic foundations and 
computational practice, building on the rich interplay between algebraic 
structures and computational optimization.

\subsection*{Acknowledgements}

The authors express their sincere gratitude to 
\textbf{Dr.~D.~Madhusudhana Rao} for his continued guidance, constructive 
criticisms, and foundational insights that shaped the algebraic and 
optimization framework developed in this work.  
The first author gratefully acknowledges the research environment and 
institutional support provided by 
\textbf{Dr.~Ramachandra R.\ K.}, Principal, 
Government College (Autonomous), Rajahmundry, which enabled the successful 
completion of this study.  

Both authors thank the Departments of Mathematics of 
Government College (Autonomous), Rajahmundry, and 
Acharya Nagarjuna University, Guntur, for their academic support throughout 
the preparation of this manuscript.


\subsection*{Ethics Statement}

This research is entirely theoretical and does not involve human participants, 
animal subjects, personal data, biological materials, or any form of 
experiment requiring institutional ethics approval.  
No external datasets requiring ethical clearance were used.  
All mathematical models, proofs, and algorithms were developed in accordance 
with the ethical guidelines and academic integrity standards of the authors’ 
institutions.


\subsection*{Conflict of Interest}

The authors declare that they have \emph{no known competing financial interests},  
personal relationships, or professional affiliations that could have appeared 
to influence the work reported in this manuscript.


\subsection*{Funding}

This research received \emph{no specific grant} from any funding agency in the 
public, commercial, or not-for-profit sectors.  
All work was carried out with institutional support from 
Government College (Autonomous), Rajahmundry, and 
Acharya Nagarjuna University, Guntur.


\subsection*{Author Contributions}

The authors contributed to this work as follows:
\begin{itemize}
    \item \textbf{Chandrasekhar Gokavarapu:}  
    Conceptualization; formulation of the TTGS framework; development of 
    ternary dynamic-programming algorithm; main theoretical results; 
    drafting and revision of the manuscript.
    
    \item \textbf{D.~Madhusudhana Rao:}  
    Supervision; validation of algebraic structures; refinement of theoretical 
    arguments; critical review and editing; guidance on mathematical coherence 
    and research direction.
\end{itemize}

Both authors read and approved the final manuscript.


\subsection*{Data Availability}

No datasets were generated or analysed in this study.  
All results are theoretical and self-contained within the manuscript.  
Any additional algebraic identities, proofs, or implementation details can be 
provided by the authors upon reasonable request.


\bibliographystyle{plain}

\end{document}